\documentclass[10pt]{amsart}
\usepackage{amscd,graphicx,epsfig}
\usepackage[colorlinks=true, pdfstartview=FitV, linkcolor=blue, citecolor=blue, urlcolor=blue]{hyperref}
\usepackage{amsmath}
\usepackage{amssymb}
\usepackage{comment}
\usepackage{amsmath}
\usepackage{mathtools}
\usepackage{hyperref}
\usepackage{subcaption}
\usepackage{graphicx}
\usepackage{tikz}
\usepackage{tikz-cd}
\usetikzlibrary{matrix,arrows,positioning,calc,chains}
\usepackage{enumitem}

\title[Characterization of $n$-dimensional normal affine  $\SL_n$-varieties]{Characterization of $n$-dimensional normal affine $\SL_n$-varieties}
\author{ANDRIY REGETA}
\address{\noindent Institut f\"{u}r Mathematik, Friedrich-Schiller-Universit\"{a}t Jena, \newline
\indent  Jena 07737, Germany}
\email{andriyregeta@gmail.com}

\thanks{The author was supported by the SNF (Schweizerischer National\-fonds), project number P2BSP2\_165390.}

\newtheorem{theorem}{Theorem}
\newtheorem{corollary}{Corollary}
\newtheorem{lemma}{Lemma}
\newtheorem{proposition}{Proposition}
\newtheorem{definition and proposition}{Definition and proposition}

\theoremstyle{definition}
\newtheorem{definition}{Definition}
\newtheorem{remark}{Remark}

\newcommand{\name}[1]{\textsc{#1\/}}

\newcommand{\spec}{\operatorname{Spec}}

\newcommand{\Der}{\mbox{Der}}
\renewcommand{\d}{{\partial}}

\DeclareMathOperator{\J}{J}

\DeclareMathOperator{\Ve}{Vec}

\DeclareMathOperator{\Lie}{Lie}
\DeclareMathOperator{\Aff}{Aff}
\DeclareMathOperator{\Aut}{Aut}
\DeclareMathOperator{\SAut}{SAut}

\DeclareMathOperator{\GL}{GL}
\DeclareMathOperator{\SL}{SL}

\DeclareMathOperator{\PSL}{PSL}

\DeclareMathOperator{\SO}{SO}
\DeclareMathOperator{\OO}{O}

\DeclareMathOperator{\jac}{jac}
\DeclareMathOperator{\U}{U}

\def \itt #1,#2:{\medskip\item[$\bullet$] %
     page\ \ignorespaces#1, line\ \ignorespaces#2:\ \ignorespaces}

\begin{document}

\begin{abstract}
We show that any normal  irreducible affine $n$-dimensional $\SL_n$-variety $X$ is determined by its automorphism group seen as an ind-group  in the category of normal irreducible affine varieties. In other words,
if $Y$ is an irreducible affine normal algebraic variety such that $\Aut(Y) \simeq \Aut(X)$ as an ind-group, then $Y \simeq X$ as a variety.  If we drop the condition of normality on $Y$,  then this statement fails.
In case $n \ge 3$,  the result above  holds true if we replace $\Aut(X)$ by $\U(X)$, where $\U(X)$
is the subgroup of $\Aut(X)$ generated by all  one-dimensional  unipotent subgroups. In dimension $2$ we have some  interesting exceptions.
\end{abstract}

\maketitle


\section{Introduction and Main Results}

Our base field  is the field of complex numbers $\mathbb{C}$. For an affine variety $X$ the automorphism group $\Aut(X)$ has the structure of an ind-group. We will shortly recall the basic definitions and results in  Section  \ref{Preliminaries}. The classical example is $\Aut(\mathbb{A}^n)$, $n>1$, the group of automorphisms of the affine $n$-space $\mathbb{A}^n$.
Recently, \name{Hanspeter Kraft} proved the following  result which shows that the affine $n$-space is determined by its automorphism group (see \cite{Kr15}).

\begin{theorem}\label{main0}
Let $Y$ be a connected affine variety. If $\Aut(Y) \simeq \Aut(\mathbb{A}^n)$ as ind-groups, then 
$Y \simeq \mathbb{A}^n$ as varieties. 
\end{theorem}

Note that this result was generalised in \cite{CRX19} (see also \cite{KRvS19} and in a similar spirit, see \cite[Theorem 1]{LRU20}) where the authors proved Theorem  \ref{main0} under a weaker condition, namely, that groups $\Aut(Y)$ and $\Aut(\mathbb{A}^n)$ are isomorphic only as abstract groups.
Moreover,
recently  Theorem  \ref{main0} was generalized in \cite[Theorem 1.4]{LRU18} (see also \cite[Main Theorem 1]{RvS19}) where it was proved that an affine toric variety different from the algebraic torus is determined by its automorphism group seen as an ind-group in the category of normal affine irreducible varieties.
If we drop the normality condition in \cite[Theorem 1.4]{LRU18}, the situation changes. In this paper we show that for ``most''
$n$-dimensional 
affine normal varieties $X$ endowed with a non-trivial regular $\SL_n = \SL_n(\mathbb{C})$-action, there are infinitely many affine varieties $Y$ such that
$\Aut(Y) \simeq \Aut(X)$ as an ind-group and we classify all such $Y$.

  Let $d > 1$.  Consider the action of $\mu_d =  \{ \xi \in \mathbb{C}^{*} | \xi^d = 1 \}$
on $\mathbb{A}^n$ by scalar multiplication and denote by $ A_{d,n}$ the quotient of $\mathbb{A}^n$ by $\mu_d$. Note that $A_{d,n}$ is normal.
Denote also by
$\pi\colon \mathbb{A}^n \to A_{d,n}$ the quotient map.  This means that $A_{d,n}$ is an affine variety with coordinate ring
\[ \mathcal{O}(A_{d,n}) = \mathbb{C}[x_1,\dots,x_n]^{\mu_d}= \bigoplus_{k=0}^{\infty} \mathbb{C}[x_1,  \dots,x_n]_{dk},\]
the algebra of invariants,  where $\mathbb{C}[x_1,\dots,x_n]_{dk}$  denotes the homogeneous polynomials of degree $dk$. 
Note that $A_{d,n}$ is indeed an orbit space, because $\mu_d$ is finite. For $d > 1$, $A_{d,n}$ has an isolated singularity in $\pi(0)$ and $\pi$ induces   an \'{e}tale covering  $\mathbb{A}^n  \setminus \{ 0 \} \to A_{d,n} \setminus \{ p(0) \}$   with Galois group $\mu_d$.

\begin{remark}
We will see in  Lemma \ref{SL2} and Proposition \ref{SLn-action} that any affine normal variety endowed with a regular non-trivial $\SL_n$-action is isomorphic to either $\SL_2/T$, $\SL_2/N$ or to $A_{d,n}$ for some $d \in \mathbb{N}$, where $T \subset \SL_2$ is the standard subtorus and $N\subset \SL_2$ is the normalizer of $T$. This implies that Theorem \ref{main2} and Theorem \ref{main3} below indeed provide the characterization of $n$-dimensional normal affine $\SL_n$-varieties.
\end{remark}

Consider the affine variety $A_{d,n}^s$ with coordinate ring 
\[ \mathcal{O}(A_{d,n}^s) = \mathbb{C} \oplus \bigoplus_{k=s}^{\infty} \mathbb{C}[x_1,\dots,x_n]_{dk} \subset \mathcal{O}(A_{d,n}), \; s \ge 1.\]
Then the induced morphism $\eta \colon A_{d,n} \to A_{d,n}^s$ is the normalization and has the property that the induced map 
$\eta' \colon A_{d,n} \setminus \{ \star \} \xrightarrow{\sim} 
A_{d,n}^s \setminus \{  \star \}$ is an isomorphism, where $\star$ denotes the points corresponding to the homogeneous maximal ideals. In fact, $\eta$ is $\SL_n$-equivariant, and $A_{d,n} \setminus \{ \star \}$
is an $\SL_n$-orbit.  We prove the following result.

\begin{theorem}\label{main2}
Let $X$ be an irreducible affine variety. Then   $\Aut(X)$ and  $\Aut(A_{d,n})$ are isomorphic as ind-groups if and only if
$X \simeq A_{d,n}^s$ as a variety for some $s \in \mathbb{N}$.  
\end{theorem}

Theorem \ref{main2} and
the following result shows that
$\SL_2/T$ and $\SL_2/N$ are the only affine $n$-dimensional $\SL_n$-varieties (except $\mathbb{A}^n$) that are determined by their automorphism groups in the category of affine irreducible varieties.

\begin{theorem}\label{main3}
Let $X$ be an irreducible variety such that $\Aut(X) \simeq \Aut(\SL_2/T)$ respectively $\Aut(X) \simeq \Aut(\SL_2/N)$ as ind-groups. Then $X \simeq \SL_2/T$ respectively $X \simeq \SL_2/N$
as varieties. 
\end{theorem}

For an affine variety $X$ we  denote by $\U(X) \subset \Aut(X)$  the subgroup generated by  the one-dimensional unipotent subgroups. We do not know whether $\U(X)$ 
has the structure of an ind-subgroup (i.e., whether $\U(X) \subset \Aut(X)$ is closed). That is why we introduce the definition of an \emph{algebraic isomorphism}.  This is  an isomorphism $\phi \colon \U(X) \xrightarrow{\sim} \U(Y)$ such that for any  subgroup $U \subset \U(X)$, where  $U$ is a closed one-dimensional unipotent subgroup of  $\Aut(X)$,  
the image $\phi(U) \subset \Aut(Y)$ is a closed one-dimensional unipotent subgroup and $\phi|_{U}\colon U \xrightarrow{\sim} \phi(U)$
is an isomorphism of algebraic groups.

\begin{theorem}\label{main4}
Let $X$ be  $A_{d,n}$, $\SL_2/T$ or $\SL_2/N$  and  $Y$ be an irreducible affine variety. Assume that there is an  algebraic  isomorphism $\U(X) \xrightarrow{\sim} \U(Y)$.  Then

(a) if $X \simeq A_{2,2}$, then $Y \simeq \SL_2/T$ or $Y \simeq A^s_{2,2}$ for some $s \in \mathbb{N}$,

(b) if $X \simeq \SL_2/T$, then $Y \simeq \SL_2/T$ or $Y \simeq A^s_{2,2}$ for some $s \in \mathbb{N}$,

(c) if   $X \simeq A_{4,2}$, then   $Y \simeq \SL_2/N$ or $Y \simeq A^s_{4,2}$ for some $s \in \mathbb{N}$,

(d) if   $X \simeq \SL_2/N$, then $Y \simeq \SL_2/N$ or $Y \simeq A^s_{4,2}$ for some $s \in \mathbb{N}$,

(e) if $X = A_{d,n}$, where $(d,n) \notin \{ (2,2), (2,4) \}$, then $Y \simeq A^s_{d,n}$ for some $s \ge 1$.
\end{theorem}

\subsection*{Acknowledgement:} The author  thanks  both referees for important remarks and suggested improvements. The  author would also like to thank
\name{Hanspeter Kraft} for his support  while  writing  this paper. Finally,
the author  is grateful to \name{Michel Brion} who suggested  a number of  improvements and \name{Mikhail Zaidenberg} for useful discussions. 

\section{Preliminaries}\label{Preliminaries}
  
    \subsection{Ind-groups}
The notion of an ind-group goes back to Shafarevich who called such objects \emph{infinite dimensional groups} (see \cite{Sh66}). We refer to  \cite{Kum02} and  \cite{FK16}  for basic notions in this context.

\begin{definition}
By an \emph{ind-variety}  we mean a set $V$ together with an ascending filtration $V_0 \subset V_1 \subset V_2 \subset \dots \subset V$ such that the following holds:

(1) $V = \bigcup_{k \in \mathbb{N}} V_k$;

(2) each $V_k$ has the structure of an affine algebraic variety;

(3) for all $k \in \mathbb{N}$ the subset $V_k \subset V_{k+1}$ is closed in the Zariski-topology.
\end{definition}

A morphism from an ind-variety $V = \bigcup_{k \in \mathbb{N}} V_k$ to an ind-variety $W = \bigcup_{m \in \mathbb{N}} W_m$ is a map $\mbox{$\phi \colon V \to W$}$  such that for any $k$ there is an $m$ such that 
$\phi(V_k) \subset W_m$ and  such  that the induced map $V_k \to W_m$ is a morphism of algebraic varieties. \emph{Isomorphisms} of ind-varieties are defined in the obvious way.

Two filtrations $V = \bigcup_{k \in \mathbb{N}} V_k$ and $V = \bigcup_{k \in \mathbb{N}} V'_{k}$ are called \emph{equivalent} if for every $k$ there is an $m$ such that $V_k \subset V'_{m}$ is a closed
 subvariety as well as $V'_{k} \subset V_m$. 
 
An ind-variety $V$ has a natural topology: a subset $S \subset V$ is open, (resp. closed), if $S_k = S \cap  V_k \subset V_k$ is open, (resp. closed), for all $k$.  
A locally closed subset $S \subset V$ has the induced structure of an ind-variety. It is called an \emph{ind-subvariety}.
A subset $S \subset V$ that is a closed subset of some $V_k$ is called an \emph{algebraic subset}.

The product of two ind-varieties is defined in the usual way. This allows to give the following definition.

\begin{definition}
An ind-variety $G$ is said to be an \emph{ind-group} if the underlying set $G$ is a group such that the map $G \times G \to G$, 
 $(g,h) \mapsto gh^{-1}$, is a morphism.
\end{definition}

An ind-group $G$ is called \emph{connected} if for every $g \in G$ there is an irreducible curve $C$ and a morphism $C \to G$ whose image contains the neutral element $e$ and $g$.

A closed subgroup $H$ of $G$ (i.e., $H$ is a subgroup of $G$ and is a closed subset) is again an ind-group under the  closed ind-subvariety structure on $G$.
A closed subgroup $H$ of an ind-group $G$ is called an \emph{algebraic subgroup} if and only if $H$ is an algebraic subset of $G$.

\begin{proposition}[{\cite[Theorem 0.3.1]{FK16}}]
\label{ind-group}
Let $X$ be an affine variety. Then $\Aut(X)$ has the structure
		of an  ind-group such that for any algebraic group $G$, there is a correspondence between regular $G$-actions on $X$ and ind-group homomorphisms $G \to \Aut(X)$. 
\end{proposition}
If $G$ is an algebraic group acting regularly and faithfully on $X$, then, by Proposition \ref{ind-group}, we can consider $G$ as an algebraic subgroup of $\Aut(X)$. We will often switch between these two points of view.

\subsection{Locally nilpotent derivations and $\mathbb{G}_a$-actions}\label{LND}
	Additive group actions on affine varieties can be described by a
	certain kind of derivations. We recall some of the basics here (see
	\cite{Fre06} for details). Let $\lambda\colon\mathbb{G}_a \to \Aut(X)$ be a
	$\mathbb{G}_a$-action on an affine variety $X$. Such an action induces a
	derivation on the level of regular functions $\mathcal{O}(X)$ by
	\[\delta_\lambda\colon\mathcal{O}(X)\rightarrow\mathcal{O}(X),\qquad
	f\mapsto \left[\frac{d}{ds}\lambda{(s)}^*(f)\right]_{s=0},
	\]
   where $\mathbb{G}_a=\spec(\mathbb{C}[s])$.
	We have that for every $f\in \mathcal{O}(X)$ there exists an
	$k\in \mathbb{N}$ with $\delta_\lambda^k(f)=0$. Derivations that have such a
	property are called {\it locally nilpotent}. Moreover, every
	$\mathbb{G}_a$-action on $X$ arises from  a certain locally nilpotent derivation
	$\delta$ and the $\mathbb{G}_a$-action
	$\lambda_\delta\colon\mathbb{G}_a\times X\rightarrow X$ is obtained from $\delta$
	via
	\[\left(\alpha_\delta(s)\right)^*\colon\OO(X)\rightarrow \OO(X)[s],
	\qquad
	f\mapsto\exp(s\delta)(f):=\sum_{i=0}^\infty\frac{s^i\delta^i(f)}{i!}\,.\]

\section{Automorphisms}

\begin{proposition}\label{prop}
Let $\pi \colon \mathbb{A}^n \to A_{d,n}$. Then
every automorphism of $A_{d,n}$ lifts to an automorphism of $\mathbb{A}^n$ which commutes with each element of $\mu_d$.
\end{proposition}
\begin{proof}
The quotient map $\pi \colon \mathbb{A}^n \to A_{d,n}$ induces a natural embedding of $\mathcal{O}(A_{d,n})$
into $\mathcal{O}(\mathbb{A}^n) = \mathbb{C}[x_1,\dots,x_n]$. So, we assume that 
$\mathcal{O}(A_{d,n})$ is a subring of $\mathbb{C}[x_1,\dots,x_n]$ and equals
$\bigoplus_{k=0}^{\infty} \mathbb{C}[x_1, \dots,x_n]_{dk}$, where $\mathbb{C}[x_1,\dots,x_n]_{dk}$  denotes the homogeneous polynomials of degree $dk$.
Let $\phi  \in \Aut(A_{d,n})$. First we claim that $p_i=\phi^*(x_i^d)$ and $p_j=\phi^*(x_j^d)$ are coprime in $\mathbb{C}[x_1,\dots,x_n]$, where $i \neq j$ and $\phi^*$ is the pull-back of $\phi$. 
Let $p$ be a  common factor of $p_i$ and $p_j$. Then $\tilde{p} = \prod_{g \in \mu_d} gp$  divides  $p_i^d$ and $p_j^d$. By construction it is clear that $\tilde{p} \in \mathcal{O}(A_{d,n})$. Hence, $(\phi^*)^{-1}(\tilde{p})$ is a common factor of $(\phi^*)^{-1}(p_i^d) = x_i^{d^2}$ and $(\phi^*)^{-1}(p_j^d) = x_j^{d^2}$. Therefore, $\tilde{p} \in \mathbb{C}$
and then, $p \in \mathbb{C}$.

  We have  $$ \phi^*(x_i^d)  \phi^*((x_j^{d})^{d-1} )  =\phi^*(x_i^d x_j^{d(d-1)} ) = \phi^*(x_i x_j^{d-1})^{d}$$ which means that $p_i p_j^{d-1} = q^d$ for some $q \in \mathcal{O}(A_{d,n})$. Because $p_i$ is coprime with $p_j$, it follows that  $p_i = q_i^d$ for some $q_i \in \mathbb{C}[x_1,\dots,x_n]$.

Since for an automorphism $\phi\colon A_{d,n} \to A_{d,n}$ we have that $\phi^*(x_i^d) = q_i^d$ for some $q_i \in \mathbb{C}[x_1,\dots,x_n]$, we  define the morphism 
\[ \tilde{\phi} =(q_1,\dots,q_n)\colon \mathbb{A}^n \to \mathbb{A}^n \]
given by the map
\[ (x_1,\dots,x_n)\mapsto  (q_1,\dots,q_n). \]
If $\phi$ is the identity automorphism, then
the restriction of $\tilde{\phi}^*$ to $\mathcal{O}(A_{d,n})$ is the identity and we have that $q_i^d = x_i^d$ which implies that $q_i = w_i x_i$ for some $w_i \in \mathbb{C}^*$, $w_i^d = 1$.  In this case $\tilde{\phi}$ is an automorphism, i.e., the trivial automorphism of $A_{d,n}$ lifts to an automorphism of $\mathbb{A}^n$
which we denote by $\Delta(w_1,\dots,w_n)$.  

Let now  $\theta\colon  A_{d,n} \to A_{d,n}$ be the inverse automorphism of $\phi \in \Aut(A_{d,n})$. Since $\phi \circ \theta$ is the trivial automorphism of $A_{d,n}$, it lifts to an automorphism $\Delta(w_1,\dots,w_n)$ of $\mathbb{A}^n$ and so $\tilde{\phi}$ is an automorphism with the inverse $\tilde{\theta} \circ \Delta(w_1,\dots,w_n)^{-1}$.
To finish the proof we need to show that 
$\tilde{\phi}$ commutes with $\mu_d$.
Indeed, since $\tilde{\phi}^*$ preserves $\mathbb{C}[x_1,\dots,x_n]^{\mu_d}$, we have that
$q_i^d(\xi x_1,\dots,\xi x_n) = 
q_i^d(x_1,\dots,x_n)$, where $\xi \in \mu_d$. Hence, 
\[q_i(\xi x_1,\dots,\xi x_n) = \xi^l q_i(x_1,\dots,x_n)\]
for some  $l =1,\dots,d-1$.  Since polynomial $q_i$ has a linear summand  it follows that $l = 1$.
The proof follows.
\end{proof}

Let $X$ be an affine variety, $H$ be a finite group that acts faithfully on $X$ and let $\pi \colon X \to X  / H$ be the quotient morphism.
Since $H$ acts faithfully, $H$ naturally embeds into $\Aut(X)$ and we identify $H$ with its image in $\Aut(X)$.  
Denote by 
$\Aut^{H}(X) \subset \Aut(X)$ the subgroup of all automorphisms of $X$ which normalize   $H$, i.e., the subgroup of those automorphisms $\phi$ such that $\phi^{-1} \circ H \circ \phi = H$.

\begin{lemma}\label{closedsubgroup}
$(a)$  There is a canonical homomorphism   of groups  $\phi \colon \Aut^H(X) \to\Aut(X /H)$,

$(b)$ if $X$   is normal and contains only finitely many fixed points of  $H$ then every $\mathbb{C}^+$-action 
on $X / H$ lifts to a 
$\mathbb{C}^+$-action on $X$. 
\end{lemma}
\begin{proof}
$(a)$
Let $h \in H$, $f \in \mathcal{O}(X)^H$ and $\phi \in \Aut^{H}(X)$. Then $\phi^* \colon \mathcal{O}(X)  \stackrel{\sim}{\rightarrow}  \mathcal{O}(X)$ is an isomorphism and 
 $$h (\phi^{*}(f)) = \phi^{*} ((\phi^{*})^{-1} \circ h \circ \phi^{*}) (f)  = (\phi^{*} \circ h') (f) = \phi^{*}(f)$$ for some $h' \in H$. Therefore
$\phi^{*}(f) \in \mathcal{O}(X)^H$, which means that $\phi$ induces an automorphism of $X / H$.

$(b)$ follows from \cite[Theorem 1.3]{MM09}.
\end{proof}

Let us recall that a closed algebraic subgroup $U$ of $\Aut(X)$ is a $1$-dimensional unipotent subgroup if $U \simeq \mathbb{C}^+$.

\begin{proposition}\label{filtration}
The  homomorphism   $\phi_d\colon \Aut^{\mu_d}(\mathbb{A}^n) \to \Aut(A_{d,n})$ is surjective
with kernel $\mu_d$. Moreover, every  $1$-dimensional   unipotent subgroup of $\Aut(A_{d,n})$
is the image of some  $1$-dimensional   unipotent subgroup of $ \Aut^{\mu_d}(\mathbb{A}^n)$.
\end{proposition}
\begin{proof}
The surjectivity of $\phi_d$ follows from Proposition \ref{prop}.  The last claim of the statement follows from 
Lemma \ref{closedsubgroup} (b). What remains is to compute the kernel of $\phi_d$.

It is clear that
$$
\Aut^{\mu_d}(\mathbb{A}^n) = \{ f = (f_1,\dots,f_n) \in \Aut(\mathbb{A}^n) \mid f_i \in \bigoplus_{k=0}^{\infty} \mathbb{C}[x_1,\dots,x_n]_{kd+1},
i=1,\dots,n \}.
$$
Now let $f=(f_1,\dots,f_n) \in \Aut^{\mu_d}(\mathbb{A}^n)$ be such that the map $f'$ induced by   $f$ on $\mathbb{A}^n/\mu_d$ is the identity.
This means that $f'$ acts trivially on $$ \mathcal{O}(\mathbb{A}^n/\mu_d) = \mathbb{C} \oplus \bigoplus_{k \ge 1} \mathbb{C}[x_1,\dots,x_n]_{kd}.$$ 
Hence, $f'(x_i^d) = x_i^d$ for any $i$ which implies that $f = ( \xi_1 x_1,\dots,\xi_n x_n)$, where $\xi_i^d = 1$ for $i = 1,\dots,n$.
In particular, 
$f'(x_i^{d-1} x_j) = x_i^{d-1} x_j$ which implies that  $\xi_i^{d-1}  \xi_j = 1$ for any $i,j$. Because $\xi_i^{d-1} \xi_i = 1$
we conclude that $\xi_i = \xi_j$.  
The claim follows.
\end{proof}

\section{Root subgroups}

Let $G$ be an ind-group, and let $T \subset G$ be a closed torus.

\begin{definition}
A closed subgroup $U \subset G$ isomorphic to $\mathbb{C}^{+}$ and normalized by $T$ is called a \emph{root subgroup} with respect to $T$. The \emph{character} of $T$ on 
$\Lie U \simeq\mathbb{C}$ i.e.,  the algebraic action of $T$ on $\Lie U$  is called the \emph{weight character} of $U$.
\end{definition}

Let $X$ be an affine variety and consider a nontrivial algebraic action of $\mathbb{C}^+$ on $X$, given by $\lambda \colon \mathbb{C}^{+} \to \Aut(X)$. If $f \in \mathcal{O}(X)$ is $\mathbb{C}^{+}$-invariant, then the \emph{modification} $f \cdot \lambda$ of $\lambda$ is defined in the    following way:
$$(f \cdot \lambda)(s)x =\lambda(f(x)s)x$$ 
for $s \in \mathbb{C}$ and $x \in X$.
It is easy to see that this is again a $\mathbb{C}^{+}$-action. In fact,  the corresponding locally nilpotent derivation to $f \cdot \lambda$ is 
$f \delta_{\lambda}$, where $ \delta_{\lambda}$ is the  locally nilpotent derivation which correspond to $\lambda$ (see Section \ref{LND} for details). 
 It is clear that if $X$ is irreducible and $f  \neq 0$, then $f \cdot \lambda$  and $\lambda$ have the same invariants.
If $U \subset \Aut(X)$ is a closed subgroup isomorphic to $\mathbb{C}^{+}$ and if $f \in \mathcal{O}(X)^U$ is a $U$-invariant, then in a similar way we define  the modification $f \cdot U$ of $U$. Choose an isomorphism
$\lambda \colon \mathbb{C}^+ \to U$ and set 
$$f \cdot U = \{ (f \cdot \lambda)(s) \mid s \in \mathbb{C}^+ \}.$$ 
Note that $\Lie(f \cdot U) = f \Lie U  \subset \Ve(X)$, where  $\Ve(X)$ denotes the
Lie algebra of (algebraic) vector fields on $X$, i.e., $\Ve(X) = \Der(\mathcal{O}(X))$, the Lie algebra
of derivations of $\mathcal{O}(X)$.

If a torus $T$ acts linearly and rationally on a vector space $V$, then we call $V$ \emph{multiplicity-free} if the weight spaces $V_{\alpha}$ are all of dimension less than or equal to $1$. 

\begin{lemma} [Lemma 6.2, \cite{Kr15}] \label{Kr15}
Let $X$ be an irreducible affine variety and let $T \subset \Aut(X)$ be a torus. Assume that there exists a root 
subgroup $U \subset \Aut(X)$ with respect to $T$ such that  the $T$-module  $\mathcal{O}(X)^{U}$  is multiplicity-free. Then $\dim T \le \dim X \le \dim T +1.$
\end{lemma}

The next result is going to be of use  in the sequel and 
can be found in \cite[Theorem 1]{Lie11}. We denote  by $\SAut(\mathbb{A}^n)$ the subgroup of $\Aut(\mathbb{A}^n)$ of the following form
\[\{ f = (f_1,\dots,f_n) \in \Aut(\mathbb{A}^n) \mid \jac(f) = \det [\frac{\d f_i}{\d x_j}]_{i,j} = 1  \}\]
and by $T_n'$ a maximal subtorus of $\SAut(\mathbb{A}^n)$  of the form
\[
\{ (t_1x_1,\dots,t_nx_n) | t_i \in \mathbb{C}^{*}, t_1 \cdot \dots \cdot t_n = 1 \}. 
\]

\begin{lemma}\label{Liendo}
Let $U  \subset \SAut(\mathbb{A}^n)$ be a one-dimensional unipotent subgroup. Then $U$
is  a root subgroup with respect to $T'_n$  if and only if  $U = U_{\lambda} = \{ (x_1,\dots,x_i +cm_i,\dots,x_n) | c \in \mathbb{C} \}$,
where $m_i = x_1^{\lambda_1} \dots x_{i-1}^{\lambda_{i-1}} x_{i+1}^{\lambda_{i+1}}\dots x_n^{\lambda_{n}}$.
The  character $\xi_{\lambda}$  corresponding  to the root subgroup $U$ is the following: $\xi_{\lambda}\colon T'_n \to \mathbb{C}^{*}$,
 $t = (t_1,\dots,t_n) \mapsto t_i t_1^{-\lambda_1}\dots \hat{t}_{i} \dots t_n^{-\lambda_n}$.
\end{lemma}

\begin{remark}\label{char}
The  last lemma can also be expressed in the following way $($see \cite[Remark 2]{KS13}$)$:  there is a bijective correspondence between the $T'_n$-stable one-dimensional unipotent subgroups $U \subset \Aut(\mathbb{A}^n)$ and the characters of $T'_n$ of the form  $\lambda = \sum_j \lambda_j  \epsilon_j$ where one $\lambda_i$ equals $1$ and the others are $\le 0$. We will denote this set of characters by $X_u(T_n')$:
$$
 X_u(T_n') =\{ \lambda = \sum \lambda_j \epsilon_j \mid   \lambda_i =1 \text{ and } \lambda_j \le 0 \; for \; j \neq i \}.
$$
 If $\lambda \in X_u(T_n')$, then $U_{\lambda}$ denotes the corresponding one-dimensional unipotent subgroup normalized by $T_n'$.
\end{remark}

\section{A special subgroup of $\Aut(X)$}\label{section5}

 For any affine variety $X$ consider the normal subgroup $\U(X)$ of $\Aut(X)$ 
generated by closed one-dimensional unipotent subgroups. The group $\U(X)$ was introduced and studied in \cite{AFK13}, where the authors called it the group of special automorphisms of $X$.
 Following  \cite{Kr15}, we introduce the notion of an algebraic homomorphism between these groups.

\begin{definition}
A homomorphism $\phi \colon \U(X) \to \U(Y)$ is algebraic if for any subgroup $U \subset \U(X)$   
such that  $U \subset \Aut(X)$ is closed,  $U  \simeq  \mathbb{C}^{+}$, the image $\phi(U) \subset \Aut(Y)$ is closed and $\phi|_U \colon U \to \phi(U)$ is a homomorphism of algebraic groups. We say that 
$\phi$ is an algebraic isomorphism if 
$\phi$ is an isomorphism of groups and 
 $\phi|_U \colon U \xrightarrow{\sim} \phi(U)$ is an isomorphism of algebraic groups.
\end{definition}

A subgroup $G \subset \U(X)$ is called algebraic if $G \subset \Aut(X)$ is a closed algebraic subgroup.
The next lemma can be found in \cite[Lemma 4.2]{Kr15}.

\begin{lemma}\label{roothomo}
Let $\phi \colon \U(X) \to \U(Y)$ be an algebraic homomorphism. Then, for any algebraic subgroup $G \subset \U(X)$ generated by
one-dimensional unipotent subgroups of $\Aut(X)$, the image $\phi(G)$ is an  algebraic subgroup of $\U(Y)$ and $\phi |_G \colon G \to \phi(G)$ 
is a homomorphism of algebraic groups.
\end{lemma}

Let $X$ be an affine variety and let $\eta\colon \tilde{X} \to X$ be a normalization map. It is well-known that any automorphism of $X$ lifts uniquely to the automorphism of $\tilde{X}$. Indeed,
for a given automorphism $\phi\colon X \to X$, the composition $\phi\circ \eta \colon \tilde{X} \to X$ is a morphism, which by the universal property of normalization factors through a morphism $\tilde{\phi} \colon \tilde{X} \to \tilde{X}$ such that $\phi \circ \eta = \eta \circ \tilde{\phi}$. It remains to argue that $\tilde{\phi}$ is an automorphism. But for the same reason, the inverse $\phi^{-1}$ lifts to an automorphism $\psi\colon \tilde{X} \to \tilde{X}$. Since $\eta\colon \tilde{X} \to X$ is birational, the compositions $\psi \circ \tilde{\phi}$ and $\tilde{\phi} \circ \psi$ are equal to the identity on a dense open subset of an irreducible variety, hence everywhere. This shows that $\eta$ induces a well-defined injective homomorphism $\tilde{\eta} \colon \Aut(X)  \hookrightarrow \Aut(\tilde{X})$. Moreover, in \cite[Proposition 12.1.1]{FK16} it is proved that $\tilde{\eta}$ is a \emph{closed immersion} of ind-groups, i.e.,  $\tilde{\eta}(\Aut(X)) \subset \Aut(\tilde{X})$ is a closed subgroup and $\tilde{\eta}$ induces the isomorphism of ind-groups $\Aut(X) \xrightarrow{\sim} \tilde{\eta}(\Aut(X))$. 
Hence, we have the following statement.

\begin{lemma}\label{closedembedding}
Let $X$ be an irreducible affine variety, and let $\eta \colon \tilde{X} \to X$ be its normalization. Then every automorphism of $X$ lifts uniquely to an automorphism of $\tilde{X}$  and the induced map $\tilde{\eta} \colon \Aut(X)  \hookrightarrow \Aut(\tilde{X})$  is a closed immersion of ind-groups.
\end{lemma}

\begin{proposition}\label{SLn-action}
Let $n \ge 3$ and let $X$ be an $n$-dimensional irreducible affine variety endowed with a non-trivial $\SL_n$-action. Then 
$$\mathcal{O}(X) = \sum_{i = 1,\dots,l} \bigoplus_{k \ge 0}    \mathbb{C}[x_1,\dots,x_n]_{k d_i}$$ for some $d_1,\dots,d_l \in \mathbb{N}$, where $(d_1,\dots,d_l) = d$ and the normalization of $X$ is isomorphic to $A_{d,n}$. The same holds when $n=2$ and the normalization of $X$ is $A_{d,2}$ for some $d \in \mathbb{N}$. 
\end{proposition}
\begin{proof}
 First, let $n \ge 3$.  If    $X$ is normal,  then  from \cite[Theorem 1.6 and Proposition 4.4(2), see also Example 4.5]{KRZ17} it follows that  $X \simeq A_{d,n}$ for some $d \in \mathbb{N}$.
It is well known that  $$\mathcal{O}(A_{d,n}) = \bigoplus_{k=0}^{\infty} \mathbb{C}[x_1,\dots,x_n]_{kd}$$ is a direct sum of irreducible pairwise non-isomorphic  $\SL_n$-modules $\mathbb{C}[x_1,\dots,x_n]_{kd}$.

  Now, consider any $n$-dimensional irreducible affine variety $X$ endowed with a non-trivial $\SL_n$-action and  a normalization morphism $\eta \colon A_{d,n} \to X$.
Since any  $\SL_n$-action on $\mathcal{O}(X)$ lifts to an $\SL_n$-action on $\mathcal{O}(A_{d,n})$, it follows 
that $\mathcal{O}(X)$ is a $\SL_n$-submodule of $\mathcal{O}(A_{d,n})$ and therefore
$$\mathcal{O}(X) = \bigoplus_{k \in \Omega} \mathbb{C}[x_1,\dots,x_n]_{kd},$$ where $\Omega$ is a submonoid of $\mathbb{N}$ under addition.
Since $\mathcal{O}(X)$ is finitely generated, $\Omega \subset \mathbb{N}$ is a finitely generated submonoid i.e., there exist 
$d_1,\dots,d_l \in \mathbb{N}$ such that
$\Omega = d_1 \mathbb{N} + \dots + d_l \mathbb{N}$. 
The claim follows.
\end{proof}


\section{$2$-dimensional case}\label{Section6}
\subsection{Two dimensional normal affine $\SL_2$-surfaces}

The next result can be found in \cite{Pop73}, \S 3 (see also \cite{Giz71} and \cite{Kr84}, \S 4).  
\begin{lemma}\label{SL2}
Let $X$ be an affine normal irreducible variety of dimension two endowed with a non-trivial $\SL_2$-action.
Then  $X$ is $\SL_2$-equivariantly  isomorphic to one of the following varieties:

 (a) $A_{d,2}$ for some $d \in \mathbb{N}$, where $\SL_2$-action on $A_{d,2}$ is induced by the standard $\SL_2$-action on $\mathbb{A}^2$,

 (b) $\SL_2/T$,  where $T$  is the standard subtorus of $\SL_2$ and $\SL_2$ acts on  $\SL_2/T$ by left multiplication, 

 (c) $\SL_2/N$, where $N$ is the normalizer of $T$ and $\SL_2$ acts on $\SL_2/N$ by left multiplication. 
\end{lemma}

The $\SL_2$-action on $\SL_2/T$ and on $\SL_2/N$ from Lemma above is transitive. 
The $\SL_2$-variety $A_{d,2}$ is the union of a fixed point and the orbit $(\mathbb{A}^2 \setminus \{ 0 \})/\mu_d \simeq  \SL_2/U_d$,
where  $\mu_d$ acts by scalar multiplication on $\mathbb{A}^2 \setminus \{ 0 \}$  and $$U_d
=\left \{ 
\begin{bmatrix}
    \xi  & t  \\
    0 & \xi^{-1}  \\
    \end{bmatrix} \mid t \in \mathbb{C}, \xi \in \mathbb{C}^*, \xi^d = 1 \right \}.$$ Moreover, any closed subgroup of $\SL_2$ of codimension 
   less than or equal to $2$ is conjugate to either $T$, or $N$, or $U_d$ for some $d \ge 1$, or $B =  \left \{ 
\begin{bmatrix}
    a  & t  \\
    0 & a^{-1}  \\
    \end{bmatrix} \mid t \in \mathbb{C}, a \in \mathbb{C}^*  \right \}$ (see for example \cite[page 803]{Pop73}).

The next result can be found in \cite[III.2.5, Folgerung 3]{Kr84}. 

\begin{proposition}\label{Kr84}
If a reductive group $G$ acts on an affine variety $X$ and if the stabilizer of a point $x\in X$ contains a maximal torus, then the orbit $Gx$ is closed.
\end{proposition}


\begin{proposition}\label{orbits}
 Let $X$ be a two-dimensional $\SL_2$-variety and let $O = \SL_2 x$ be the orbit of $x \in X$. 
 Then we are in one of the following cases:

\begin{enumerate}[label=(\alph*)]
\item $x$ is a fixed point;
\item  the orbit $O$ is closed and $\SL_2$-isomorphic to $\SL_2/T$ or $\SL_2/N$;
\item $\overline{O} = O \cup \{ x_0 \}$, where $\overline{O}$ is the closure of the orbit $O$ and $x_0$ is a fixed point. Moreover, either $\overline{O} \simeq \mathbb{A}^2$ or $x_0$ is an isolated singular point.
\end{enumerate}
\end{proposition}

\begin{proof} If the stabilizer of $x$ contains a maximal torus then we are in case (a) or (b) by Proposition \ref{Kr84} and Lemma \ref{SL2}. 
Otherwise, from the classification of closed subgroups of $\SL_2$ it follows that  the stabilizer of $x$ coincides with $U_d$ for some $d \ge 1$ and
$\overline{O}$ does not contain orbits of dimension one. Hence,
   $\overline{O} = O \cup \{ x_0 \}$. It is clear that if $\overline{O}$ is singular, then $x_0$ is an isolated singular point. If 
   $\overline{O}$ is smooth, then from Lemma \ref{SL2}  it follows that  $\overline{O}$ is isomorphic to $\mathbb{A}^2$.
\end{proof}

\begin{remark}\label{SL2TP2}
Note that $\SL_2/T \simeq \mathbb{P}^1 \times \mathbb{P}^1 \setminus \Delta$, where $\Delta$ is the diagonal, and $\SL_2/N \simeq \mathbb{P}^2 \setminus C$, where $C$ is a smooth conic  (see \cite[Lemma 2]{Pop73}).
\end{remark}


\subsection{The structure of $\Aut(\SL_2/T)$}\label{structureAutSLT}
The variety  $\SL_2/T$  is isomorphic to the following  so-called \name{Danielewski} surface, i.e.,  the smooth $2$-dimensional affine quadric $ V(xz - y^2+y) \subset \mathbb{A}^3$ 
(see \cite{DP09})
and the  quotient map $\pi \colon \SL_2 \to \SL_2/T$ is given by $\begin{bmatrix}
 a & b  \\
 c & d  
\end{bmatrix} \mapsto (ab,ad,cd) $.
It is not difficult to see that  $X = V(xz + y^2 - 1)    \simeq  V(xz - y^2 + y)  \subset \mathbb{A}^3$. From now on and until end of Section \ref{Section6} we identify $\SL_2/T$ with $X= V(xz + y^2 - 1)$.


Consider the orthogonal group $\OO_3=\OO(3,\mathbb{C})$ 
associated with the quadratic form $y^2+xz$ 
generated by
$\tau \colon \mathbb{A}^3 \to \mathbb{A}^3$ given by the following map: $(x,y,z) \mapsto (-x,-y,-z)$
and by
the
 group $\SO_3=\SO(3,\mathbb{C})$ that is composed of the matrices
 \[  \frac{1}{ad-bc}
 \begin{pmatrix} a^2 & 2ab & -b^2  \\ 
 ac & ad + bc & -bd  \\ 
 -c^2 & -2cd & d^2  \\ 
 \end{pmatrix}   \text{ with } 
 \begin{pmatrix} 
 a & b   \\ 
 c & d   \\ 
 \end{pmatrix} \in \PSL_2. \]

 Following  \cite[Theorem 6]{Lam05} (see also \cite{MM90}),
 $\Aut(X)$ is the amalgamated product of  the  orthogonal  group  $\OO_3 = \SO_3 \times \langle \tau \rangle$ and  $\J \rtimes \langle \tau  \rangle$ along their intersection $C$, where 
$\J$ is the  subgroup of  $\Aut(X)$ of the automorphisms of the form
$$
(x,y,z) \mapsto (\alpha x+2 \alpha yP(z) - \alpha z P^2(z),  y-zP(z), \frac{1}{\alpha}z) ;  \; \alpha \in \mathbb{C}^*, P \in \mathbb{C}[z]. 
$$

Note that $\SO_3$ is generated by 
$\mathbb{C}^+$-actions. Define
 $\tilde{\J}$ to be the subgroup of $\J$ generated by $\mathbb{C}^+$-actions.   The subgroup of 
  $\Aut(X)$ generated by $\SO_3$ and 
   $\tilde{\J}$ coincides with the subgroup of 
  $\Aut(X)$ generated by $\SO_3$ and 
   $\J$ and is a subgroup of $\U(X)$. Moreover, because 
   $\tau$ normalizes $\langle \J, \SO_3  \rangle$
we have that 
$\Aut(X) = \langle \J, \SO_3  \rangle \rtimes \langle \tau \rangle$.
This implies that $\Aut(X)$ is not connected and $\tau \notin \Aut^
\circ(X)$.  Since the closure of $\U(X) \subset \Aut(X)$ is connected as $\U(X)$ is generated by connected subgroups
and since $\langle \J, \SO_3  \rangle \subset \U(X)$ we have that $\U(X) \subset \Aut(X)$ coincides with $\Aut^\circ(X)$ and hence  is  closed. Moreover, $\Aut(X) = \U(X) \rtimes \langle \tau \rangle$.

   The following proposition is going to be of use later  (see \cite[Corollary 8.11]{Neu48}).
\begin{proposition}\label{Neu}
  In the amalgamated product $G = A \ast_C B$ with the unified subgroup $C = A\cap B$, consider two subgroups $\tilde{A} \subset A$ and $\tilde{B} \subset B$, and let $\tilde{G} = \langle  \tilde{A}, \tilde{B} \rangle$. Assume that $\tilde{A} \cap C = \tilde{C} = \tilde{B} \cap C$. Then $\tilde{G} = \tilde{A} \ast_{\tilde{C}} \tilde{B}$. 
\end{proposition}

\begin{lemma}\label{amalgamT}
The group $\U(X)$ is the amalgamated product of $\SO_3$ and $\J$ 
 along their intersection. 
\end{lemma}
\begin{proof}
 We know that $\Aut(X)$ is the amalgamated product of 
  $\OO_3 = \SO_3 \times \langle \tau \rangle$ and  $\J \rtimes \langle \tau  \rangle$ along their intersection $C$.
  Moreover, since $\SO_3 \cap C = 
  \J \cap \SO_3 = \J \cap C$ we have by Proposition \ref{Neu} that $\langle \J, \SO_3  \rangle \subset \Aut(X)$ is the amalgamated product of $\J$ and $\SO_3$ along their intersection.  
  As $\U(X) = \langle \J, \SO_3  \rangle$
  the claim follows.
\end{proof}
 
\begin{lemma}\label{centralizer}
 The subgroup $\Aut_{\SO_3}(X) \subset \Aut(X)$ of those automorphisms that commute with $\SO_3$ is $\langle \tau  \rangle$.
 Moreover,  $X/\langle \tau \rangle \simeq \SL_2 /N$. 
 \end{lemma}
 \begin{proof}
Let $\varphi \in \Aut_{\SO_3}(X)$. 
Since $\Aut(X)$ is the amalgamated product of $\OO_3$ and $\J$ along their intersection,
we can
write $\varphi$ as the product $a_1\circ b_1 \circ \dots \circ a_k \circ b_k$, where $a_i \in \OO_3$ and $b_i \in \J$. Since $a = \varphi \circ a \circ \varphi^{-1}$ for any $a \in \SO_3$,
we have that $b_k \in \OO_3 \cap \J$ and $\varphi$ can be written as 
$\tilde{a}_1\circ \tilde{b}_1 \circ \dots \circ \tilde{a}_k$ for some $\tilde{a}_i \in \OO_3$ and $\tilde{b}_i \in \J$. Assume first that $k > 1$. Hence,  the expression
\[  \tilde{a}_1\circ \tilde{b}_1 \circ \dots \circ \tilde{a}_k \circ a \circ (\tilde{a}_1\circ \tilde{b}_1 \circ \dots \circ \tilde{a}_k)^{-1} = a \]
implies that  $\tilde{a}_k \circ a \circ \tilde{a}_k^{-1} \in \OO_3 \cap \J$. Since this should hold for any $a \in \SO_3$ we get a contradiction. Therefore, $k = 1$ and hence, $\varphi \in \OO_3$. This implies that $\Aut_{\SO_3}(X) = \langle \tau \rangle$ as the centralizer of $\SO_3$ in $\OO_3$ is $\langle \tau \rangle$.

To finish the proof we need to argue that $X/\langle \tau \rangle \simeq \SL_2 /N$. We first note that $X/\langle \tau \rangle$ is normal and since $\SL_2$-action  on $X$ is transitive, it follows that the induced action of $\SL_2$  on $X/\langle \tau \rangle$ is transitive too. Hence,  from Lemma \ref{SL2} it follows that $X/\langle \tau \rangle \simeq \SL_2/N$ as $X \not\simeq X/\langle \tau \rangle$.
\end{proof}

Note that the  subgroup $\U(X) = \Aut^\circ(X) \subset \Aut(X)$ is closed (see \cite[Lemma 6.3]{Kr15}), where $\Aut^\circ(X)$ is the neutral component of $\Aut(X)$.  Hence, $\U(X)$ is an ind-group.

\begin{proposition}\label{proposition2}
 We have the following properties.

$(a)$ All closed subgroups $S \subset \Aut(X)$ isomorphic to $\PSL_2$ are conjugate. 

$(b)$ The root subgroups with respect to a maximal torus $\tilde{T}$ of any $S \simeq \PSL_2$
are multiplicity-free with weights $1,2,3,\dots$ up to an automorphism of $\tilde{T}$.
\end{proposition}
\begin{proof}
$(a)$ Since  $\Aut(X)$ is the amalgamated
product of $\OO_3$ and $\J$  over their intersection we have that  by  \cite{Sr80}   any closed subgroup $S \subset \Aut(X)$ isomorphic to $S$  is conjugate to one of the factors $\OO_3$ or $\J$. Since all unipotent subgroups of $\J$ commute, $S$ can not be embedded into $\J$ and hence $S$ is conjugate to a subgroup of $\OO_3$, i.e., to $\SO_3$. The claim follows.
 
Now we are going to prove $(b)$. 
Let $U \subset \Aut(X)$ be a root subgroup with respect to $\tilde{T}$. This means that $\tilde{T} \ltimes U$ is an algebraic subgroup of $\Aut(X)$ and by \cite{Sr80}, $\tilde{T} \ltimes U$ is conjugate to a subgroup of either $\OO_3$ or $\J$. If $\tilde{T} \ltimes U$ is conjugate to a subgroup of $\OO_3$, then the weight of $U$ with respect to $\tilde{T}$ is either $1$ or $-1$, i.e., up to an automorphism of $\tilde{T}$ we can assume that the weight is $1$. If  $\tilde{T} \ltimes U$ is conjugate to a subgroup of $\J$, then without loss of generality we can assume that $\tilde{T} \ltimes U$ is  an algebraic subgroup  of 
$\J_{\le k}$ generated by elements of the form
$$
(x,y,z) \mapsto (\alpha x+2 \alpha yP(z) - \alpha z P^2(z),  (y-zP(z)), \frac{1}{\alpha}z) ;  \; \alpha \in \mathbb{C}^*, P \in \mathbb{C}[z]_{\le k}. 
$$
for some natural $k$, where $\mathbb{C}[z]_{\le k}$ denotes the polynomials of degree less or equal than $k$.
Moreover, since all tori in $\J_{\le k}$ are conjugate  we can assume that 
\[  \tilde{T} = \{ (tx,y,t^{-1}z) \mid t \in \mathbb{C}^*  \}.\]
By the following computation
\begin{align*}
(t x, y, t^{-1} z)  \circ  ( x+2  yP(z) &-  z P^2(z),  (y-zP(z)), z) \circ (t^{-1} x, y, t z) =  \\
 &= ( x+2 y tP(t z) -  z t^2 P^2(t z),  (y-z t P(t z)), z),
\end{align*}
it is easy to see  that
a root subgroup $U \subset \J_{\le k}$ should have the form
\[  U_i = \{  ( x+2 cyP_i(z) -  c^2z P_i^2(z),  (y-czP_i(z)), z) \mid c \in \mathbb{C}, \; P_i(z) = z^i     \}\] 
for some natural $i \le k$. Note that the root subgroup $U_i$ with respect to  $\tilde{T}$ has the
 weight $i+1$. The claim follows.
\end{proof}

\subsection{The structure of $\Aut(\SL_2/N)$.}

By Lemma \ref{centralizer}, there is an automorphism $\tau \in \Aut_{\SO_3}(X)$  and the quotient $Y = X/\langle \tau \rangle$ is isomorphic to $\SL_2 /N$.
In particular, $\mathcal{O}(Y ) = \mathcal{O}(X)^{\tau}$. An automorphism $\phi$ of $X$ descends to an automorphism on $Y$ if and only if  $\phi$ sends
$\langle \tau \rangle$-orbits	to	$\langle \tau \rangle$-orbits.	In fact,	such	an	automorphism induces the automorphism of $\mathcal{O}(X)$ that	sends $\langle \tau \rangle$ -invariant functions  of $ \mathcal{O}(X)$  to $\langle \tau \rangle$-invariant functions  of $ \mathcal{O}(X)$. 
 This condition for $\phi$ is equivalent to the condition that $\phi$ normalizes $\langle \tau \rangle$.
 Moreover,  since $\tau$ has order two, $\phi$ commutes with $\tau$.
  Recall that by $\Aut^{\langle \tau \rangle}(X)$ we denote  the  subgroup of those elements of $\Aut(X)$ that normalize $\langle  \tau \rangle$, but in this particular case  $\Aut^{\langle \tau \rangle}(X)$ is even the subgroup of those automorphisms of $\Aut(X)$ that commute with $\tau$.   
  
As we have mentioned above, $\phi \in \Aut(X)$     induces an automorphism of  $Y \simeq \SL_2/N$ if and only if $\phi \in \Aut^{\langle \tau \rangle}(X)$. On the other hand,
since $X \simeq \SL_2/T$ is simply connected and the quotient map $\pi\colon X \to X/\langle \tau \rangle = Y$ is  an  \'etale covering, 
  every automorphism $\varphi$ of $Y$ can be lifted to a continous analytical automorphism of $X$ and hence by \cite[Proposition 20]{Sr58}, $\varphi$ can be lifted to an automorphism $\tilde{\varphi}$ of $X$, i.e., $\tilde{\varphi} \in \Aut^{\langle \tau \rangle}(X)$.
  Hence, we have the surjective homomorphism $\Aut^{\langle \tau \rangle}(X) \to \Aut(Y)$ with the kernel $\langle \tau \rangle$  and so 
  \begin{equation}\label{quotient}
      \Aut(Y) \simeq \Aut^{\langle \tau \rangle}(X)/\langle \tau \rangle.
  \end{equation}

Observe that $\SO_3 \times \langle \tau \rangle$ is the subgroup of $\Aut^{\langle \tau\rangle}(X)$.
Define the subgroup $\J^{\langle \tau\rangle} \subset \Aut(X)$ of those automorphisms from $\J$ which normalize $\langle \tau\rangle$. It is not difficult to see that
$\J^{\langle \tau\rangle}$
is comprised of the following automorphisms:
\[
\{ (x,y,z) \mapsto (\alpha x+2 \alpha yP(z) - \alpha z P^2(z),  y-zP(z), \frac{1}{\alpha}z) ;  \; \alpha \in \mathbb{C}^*, P \in \bigoplus_{l=0}^{\infty} \mathbb{C}z^{2l}  \}.
\]
We have the following statement.

\begin{lemma}\label{amalgamatedproducttau}
The subgroup
$\Aut^{\langle \tau \rangle}(X) \subset \Aut(X)$ 
is the direct product of $\langle \tau \rangle$ and the amalgamated product of $\SO_3$ and $\J^{\langle \tau \rangle}$ along their intersection.
\end{lemma}
\begin{proof}
As we have mentioned above, $\Aut^{\langle \tau \rangle}(X)$ is the subgroup of those automorphisms of $\Aut(X)$ that commute with $\tau$. Assume $\phi \in \Aut(X)$ commutes with $\tau$. Since $\Aut(X)$ is the amalgamated product of 
$\SO_3 \times \langle \tau \rangle$ and $\J \rtimes \langle \tau \rangle$ one can write $\phi$ as a product $a_1\circ b_1 \circ \dots \circ a_k \circ b_k$, where $a_i \in \SO_3 \times \langle \tau \rangle$ and $b_i \in \J \rtimes  \langle \tau \rangle$.
 Further, because $\phi$ commutes with $\tau$, $\tau \phi \tau = \phi$ or equivalently,
 \[
 \tau \circ a_1\circ b_1 \circ \dots \circ a_k \circ b_k \circ \tau = a_1\circ b_1 \circ \dots \circ a_k \circ b_k.
 \]
Since $\tau$ commutes with $\SO_3$ one can rewrite this equation as follows:
  \[
 a_1\circ ( \tau \circ b_1 \circ \tau) \circ \dots \circ a_k \circ ( \tau \circ b_k \circ \tau)  = a_1\circ b_1 \circ \dots \circ a_k \circ b_k.
 \]
 From the amalgamated product structure of $\Aut(X)$ it follows that $\tau \circ b_i \circ \tau = c_i \circ b_i$ for some $c_i \in (\SO_3 \times \langle \tau \rangle) \cap (\J \rtimes \langle \tau \rangle)$. But this can happen only if $b_i$ commutes with $\tau$, i.e., $b_i \in \J^{\langle \tau \rangle} \times \langle \tau \rangle$. Therefore, $\Aut^{\langle \tau \rangle}(X)$ is generated by $\SO_3 \times \langle \tau \rangle$ and $\J^{\langle \tau \rangle} \times \langle \tau \rangle$.
 Moreover, $\Aut^{\langle \tau \rangle}(X) = \langle \SO_3,\J^{\langle \tau \rangle} \rangle \times \langle \tau \rangle$
 and by Proposition \ref{Neu}, $\langle \SO_3,\J^{\langle \tau \rangle} \rangle$
  is the amalgamated product of $\SO_3$ and $\J^{\langle \tau \rangle}$ over their intersection.  
\end{proof}

From Lemma \ref{amalgamatedproducttau} and \eqref{quotient} we have the following statement.

\begin{lemma}\label{amalgamatedSL_2/N}
The automorphism group $\Aut(Y)$ is isomorphic to the amalgamated product of
$\SO_3$ and $\J^{\langle \tau\rangle}$. In particular, $\Aut(Y) = \U(Y)$.
\end{lemma}

\begin{remark}
Lemma \ref{amalgamatedSL_2/N} can also be retrieved from   \cite[Remark 3.9]{KPZ17} and Remark \ref{SL2TP2} (see also \cite[(2.4.3;l)]{DG77}).
\end{remark}

\begin{corollary}\label{corollary2}
 We have the following properties.

$(a)$ All closed subgroups $S \subset \Aut(Y)$ isomorphic to $\PSL_2$ are conjugate. 

$(b)$ The root subgroups of $\Aut(Y)$ with respect to a maximal torus $\tilde{T}$ of any $S \simeq \PSL_2$
are multiplicity-free with weights $1,3,5,\dots$ up to an automorphism of $\tilde{T}$.   In  particular, $\U(\SL_2 /N)$ and $\U(\SL_2/T)$
are not algebraically isomorphic.
\end{corollary}
\begin{proof}
$(a)$ By Lemma \ref{amalgamatedSL_2/N}, $\Aut(Y)$
is the amalgamated product of $\SO_3$ and $\J^{\langle \tau\rangle}$ and by \cite{Sr80} any algebraic subgroup of the amalgamated product is conjugate to one of the factors. Since, $\J^{\langle \tau\rangle}$ does not contain a copy of $\PSL_2$ it follows that  $S$ is conjugate to a subgroup of $\SO_3$,  i.e., to $\SO_3$ itself.

$(b)$ Without loss of generality we can assume that $S$ equals $\SO_3$ and $\tilde{T} \subset \SO_3$ is the  subtorus of the form
\[  \{ (tx,y,t^{-1}z) \mid t \in \mathbb{C}^*  \}.\]
 Any root subgroup of $\Aut(Y)$ with respect to $\tilde{T}$ lifts to a root subgroup $U$ of $\Aut^{\langle \tau \rangle}(X)$
 (see  Lemma \ref{closedsubgroup}) with respect to the subtorus $p^{-1}(\tilde{T})^{\circ} \subset \Aut^{\langle \tau \rangle}(X)$.
 As it follows from the proof of Proposition \ref{proposition2}(b), $U$ coincides with 
  \[U_{2i} = \{  ( x+2  yP_i(z) -  z P_i^2(z),  (y-zP_i(z)), z) | P_i(z) = z^{2i}     \}\]
  for some $i \in \mathbb{N} \cup \{ 0 \}$.
  The weight of the root subgroup $U_{2i} \subset \Aut^{\langle \tau\rangle}(X)$ with respect to $p^{-1}(\tilde{T})^{\circ}$ is $2i+1$. Since 
  the kernel of $p^{-1}(\tilde{T}) \to \tilde{T}$ is trivial we have that the set of weights
  of root subgroups of $\Aut(Y)$ with respect to $\tilde{T}$ is  
  $\{ 2i+1 \mid i \in \mathbb{N} \}$. This proves the first part of the statement. 
  The second part follows because if there is an 
  algebraic isomorphism 
  $\varphi\colon \U(X) \to \U(Y)$, then $\varphi$ 
  maps root subgroups of $\U(X)$ with respect  to a subtorus $\tilde{T} \subset \U(X)$ to root subgroups of $\U(Y)$
  with respect to $\varphi(\tilde{T})$ that have the same weights. But as follows from Proposition \ref{proposition2} and the first part of this proof it is not the case. 
\end{proof}

\begin{remark}\label{centralizerSLN}
Analogously as in the proof of Lemma \ref{centralizer}, using amalgamated product structure of $\Aut(Y)$ described in Lemma \ref{amalgamatedSL_2/N}
we can show that the subgroup $\Aut_{\SO_3}(X) \subset \Aut(X)$ of those automorphisms that commute with $\SO_3$ is trivial.
\end{remark}


\subsection{On the automorphism group of $A_{d,2}$}

Recall that by Proposition \ref{filtration}, there is a surjective homomorphism $\phi_d \colon \Aut^{\mu_d}(\mathbb{A}^n) \to \Aut(A_{d,n})$ of groups. 
Consider now the maximal subtorus 
\[T_{n} = \{ (t_1x_1,\dots,t_nx_n) | t_i \in \mathbb{C}^{*} \} \subset \Aut(\mathbb{A}^n)
\]
and recall that by $T_n'$ we denote the subtorus of the form
\[
\{ (t_1x_1,\dots,t_nx_n) | t_i \in \mathbb{C}^{*}, t_1 \cdot \dots \cdot t_n = 1 \} \subset \U(\mathbb{A}^n)
\]
that has  dimension $n-1$. Then
 $T_{d,n}' = \phi_d(T'_n)$ is a maximal subtorus of $\U(A_{d,n}) \subset \Aut(A_{d,n})$.

\begin{lemma}\label{Lie+}
Let $U  \subset \Aut(A_{d,n})$ be a root subgroup with respect to $T_{d,n}'$ which has a  character $\chi$. Then $U$ lifts to a  root subgroup $\tilde{U} = (\phi_d^{-1}(U))^\circ \subset  \Aut^{\mu_d}(\mathbb{A}^n)$   with respect  to $T_n' = (\phi_d^{-1}(T_{d,n}'))^\circ$ with character   $\tilde{\chi} = \psi^*(\chi)$   such that the following diagram
$$\begin{CD} 
1 @>{}>>  \mu_{(n,d)} @>{}>> T_n' @>{\psi}>> T_{d,n}' @>{}>> 1 \\
@.  @.  @VV{\tilde{\chi}}V@VV{\chi}V \\
 @. @.  \mathbb{C}^*  @>{=}>>   \mathbb{C}^* 
\end{CD}
$$
 commutes, where $\psi = \phi_d|_{T_n'}$ and $\psi^*(\chi)$ is a pull-back of $\chi$.
\end{lemma}

\begin{proof}  From Proposition \ref{prop} it follows that any root subgroup $U$ of $\Aut(A_{d,n})$ with respect to $T_{d,n}'$ lifts
 to a unipotent subgroup $\tilde{U} = (\phi_d^{-1}(U))^\circ$ of $\Aut^{\mu_d}(\mathbb{A}^n)$. Moreover, $\tilde{U}$ is normalized by $(\phi_d^{-1}(T_{d,n}'))^{\circ} = T'_n$. 
Now, let $\tilde{u} \in  \tilde{U}$ be a non-trivial element  and $u = \phi_d(\tilde{u}) \in U$.  We have group isomorphisms
\[
 \mathbb{C}^+ \xrightarrow{\sim} \tilde{U}, \; \;
 s  \mapsto \tilde{u}(s)
\]
and 
\[
 \mathbb{C}^+ \xrightarrow{\sim} U, \; \;
 s  \mapsto u(s).
\]
Now the proof follows from the formula
\[
\phi_d(\tilde{u}(\chi\circ\psi(t) s)) = u(\chi(t) s).
\]
\end{proof}

Observe that the homomorphism $\phi_d \colon \Aut^{\mu_d}(\mathbb{A}^n) \to \Aut(A_{d,n})$ induces the homomorphism  $\tilde{\phi}_{d} \colon \U^{\mu_d}(\mathbb{A}^n) \to \U(A_{d,n})$ which has the kernel $\mu_{(n,d)}$,  where 
$\U^{\mu_d}(\mathbb{A}^n) \subset \Aut^{\mu_d}(\mathbb{A}^n)$ is a subgroup generated by $\mathbb{C}^+$-actions.

In \cite{Hausen} it is proved that any  faithful action of an $(n-1)$-dimensional torus on an $n$-dimensional toric $T_Z$-variety $Z$ is conjugate to a subtorus of the big torus $T_Z$. 
 This result is used in order to prove the following lemma. 

\begin{lemma}\label{torus}
Let $T$  be an algebraic subtorus of $\U(A_{d,n})$ of dimension $(n-1)$. Then there exists an algebraic isomorphism $F\colon \U(A_{d,n})  \stackrel{\sim}{\rightarrow} \U(A_{d,n})$ such that 
$F(T) = T_{d,n}'$.
\end{lemma}
\begin{proof}
Since $T \subset \U(A_{d,n}) \subset \Aut(A_{d,n})$ is an algebraic subtorus of dimension $n-1$ and since $A_{d,n}$ is toric, by \cite[Theorem p. 2]{Hausen} there exists $\varphi \in \Aut(A_{d,n})$ such that $\varphi \circ T \circ \varphi^{-1} \subset T_{d,n}'$. Moreover, since $\U(A_{d,n})$ is a normal subgroup of $\Aut(A_{d,n})$, $\varphi \circ T \circ \varphi^{-1} \subset T_{d,n}'$ and hence since $\dim T_{d,n}' = n-1$,
$\varphi \circ T \varphi^{-1} = T_{d,n}'$. This proves that an algebraic isomorphism $F \colon \U(A_{d,n}) \to \U(A_{d,n})$, $\psi \mapsto \varphi \circ \psi \circ \varphi^{-1}$ maps $T$ to $T_{d,n}'$.
\end{proof}

Let $Z$ be an irreducible affine variety of dimension $n \ge 2$ and $\psi \colon \U(Z)  \stackrel{\sim}{\rightarrow}  \U(A_{d,n})$   be an algebraic isomorphism. Let 
$T$  be an $(n-1)$-dimensional algebraic subtorus of $\U(Z)$. Then, after composing $\psi$
with a suitable algebraic isomorphism $F\colon \U(A_{d,n}) \to \U(A_{d,n})$ (see  Lemma 
\ref{torus}), we can assume that $\psi(T) = T_{d,n}'$.

\begin{lemma}\label{rootweights} 
Root subgroups $U$ and $\psi(U)$ have the same weight characters with respect to $T$ and $\psi(T) = T_{d,n}'$ respectively, i.e., 
if $\chi \colon T_{d,n}' \to \mathbb{C}^*$
is the weight of $\psi(U)$, then
the weight of $U$ is $\chi \circ \psi$. 
\end{lemma}
\begin{proof}
Let $U$ be a root subgroup  of $\U(Z)$ with respect to $T$ and $\Lie U = \mathbb{C}  \nu$, where $\nu$ is a generator.  Then $\psi(U)$ is the root subgroup of $\U(A_{d,n})$
with respect to $T_{d,n}'$.  The  algebraic  isomorphism  $\psi$ induces an isomorphism $d \psi^{u}_e \colon \Lie U \to \Lie \psi(U)$. Note that the action of $T$ on $U$ induces the 
action of $T$ on $\Lie U$.  Then 
\[\psi(t) \circ d\psi^u_e(\nu)\circ \psi(t^{-1})=
\chi(\psi(t)) d\psi^u_e(\nu) 
 = d\psi^u_e(\chi(\psi(t))\nu)= d\psi^u_e(\chi \circ \psi(t)\nu),
\]
where $t \in T$.
On the other hand, 
\[
 \psi(t) \circ d\psi^u_e(\nu)\circ \psi(t^{-1})= d\psi^u_e(t \circ \nu \circ t^{-1}).
\]
The claim follows. 
\end{proof}

\begin{lemma}\label{rootseven}
Let $d$ be even. Then the set of weights of root subgroups of $\Aut(A_{d,2})$ with respect to $T_{d,2}'$ is $\{\frac{kd+2}{2} | \;  k \in \mathbb{N} \cup \{ 0 \} \}$ up to an automorphism of $T'_{d,2}$. 
\end{lemma}
\begin{proof}
By Proposition \ref{filtration}, any root subgroup of $\Aut(A_{d,2})$ with respect to $T_{d,2}'$ lifts to a root subgroup of $\Aut^{\mu_d}(\mathbb{A}^2)$
 with respect to $\phi_d^{-1}(T'_{d,2})^\circ = T_2'$. By Lemma \ref{Liendo}, any root subgroup of $\Aut(\mathbb{A}^2)$ with respect to $T_2'$ is equal either to
 $$U_s = \{ (x +c y^s,y) \mid c \in \mathbb{C} \}$$
 or to
 $$\tilde{U}_l = \{ (x,y+cx^l) \mid c \in \mathbb{C} \}$$
 for some $s,l \in \mathbb{N} \cup \{ 0 \}$.  Root subgroups $U_s$ and $\tilde{U}_l$ belong to $\Aut^{\mu_d}(\mathbb{A}^2)$ if and only if $s,l \in d\mathbb{N} + 1$. The weight of the action of $T_2' = \{ (cx,c^{-1}y) \mid c \in \mathbb{C}^* \}$
 on $U_s$ by $t \circ u \circ t^{-1}$, $t \in T_2'$ and $u \in U_s$, equals $s+1$. Analogously, the weight of
 $T_2'$-action on $\tilde{U}_l$ is $-l-1$.
 Therefore, the set of weights of root subgroups of $\Aut^{\mu_d}(\mathbb{A}^2)$ with respect to $T_2'$ is $\{kd+2 | \;  k \in \mathbb{N} \cup \{ 0 \} \}$ up to an automorphism of $T'_2$. Moreover,
 since the kernel of the map $\phi_d\colon T_2' \to T_{d,2}$ is $\mu_2$ as $d$ is even, the statment follows from Lemma \ref{Lie+}.
\end{proof}

By the Jung-Van der Kulk Theorem (see \cite{Ju42} and \cite{Kul53})
$\Aut(\mathbb{A}^2) = \Aff_2 \ast_{C} \J$, where $\Aff_2$ is the group of affine transformations of $\mathbb{A}^2$, 
\[\J = \{ (ax+c, by + f(x)) | a,b \in \mathbb{C}^*, c \in \mathbb{C}, f(y) \in \mathbb{C}[x] \}\] and $C = \Aff_2 \cap \J$.  
Now, the subgroup $\Aut^{\mu_d}(\mathbb{A}^2) \subset \Aut(\mathbb{A}^2)$  contains the standard $\GL_2 \subset \Aut(\mathbb{A}^2)$ and 
\[ \J_d = \{ (ax+c, b y + f(x)) \mid a,b \in \mathbb{C}^*, c \in \mathbb{C}, f(y) \in \bigoplus_{l\ge 0} \mathbb{C}x^{ld+1} \}.\]
By \cite[Theorem 4.2]{AZ13}, $\Aut(A_{d,2}) \simeq \Aut^{\mu_d}(\mathbb{A}^2)/\mu_d$ is the amalgamated product of  $\GL_2/\mu_d$ and $\J_d/\mu_d$ along their intersection. Moreover, as we will see in Lemma \ref{amalgamAd2} below, such an amalgamated product structure induces the amalgamated product structure of $\U(A_{d,2})$.
Denote by $\tilde{\J}_d$ the subgroup of $\J_d$ of the following form:
\[  \{ (ax+c, a^{-1} y + f(x)) \mid a \in \mathbb{C}^*, c \in \mathbb{C}, f(y) \in \bigoplus_{l\ge 0} \mathbb{C}x^{ld+1} \}\]
and by $\tilde{T}_{d,2}$  the one-dimensional subtorus of $\Aut(A_{d,2})$ induced by the $\mathbb{C}^*$-action on $\mathbb{A}^2$ given by the maps  $\{ (x,y) \mapsto (cx,y)\mid c \in \mathbb{C}^* \}$.
 We have the following statement.
\begin{lemma}\label{amalgamAd2}
The group $\Aut(A_{d,2})$ is the semidirect product $\U(A_{d,2}) \rtimes \tilde{T}_{d,2}$. 
Moreover, 
$\U(A_{d,2})$ is the amalgamated product of $\SL_2/(\mu_d \cap \SL_2)$
and $\tilde{\J}_d/(\mu_d \cap \tilde{\J}_d)$ along their intersection.
\end{lemma}
\begin{proof}
As $\Aut(A_{d,2}) \simeq \Aut^{\mu_d}(\mathbb{A}^2)/\mu_d$,  it is clear that $\Aut(A_{d,2})$ is generated by $\U(A_{d,2})$ and $\tilde{T}_{d,2}$. Further, the subgroup $\U(A_{d,2}) \subset \Aut(A_{d,2})$ is normal and the subgroups $\U(A_{d,2})$ and $\tilde{T}_{d,2}$ do not intersect. Indeed, if the
subgroups $\tilde{T}_{d,2}$ and $\U(A_{d,2})$ have a non-trivial intersection, then as $\Aut(A_{d,2}) \simeq \Aut^{\mu_d}(\mathbb{A}^2)/\mu_d$, 
the subgroups $\{ (x,y) \mapsto (cx,y)\mid c \in \mathbb{C}^* \}$ and 
$\U(\mathbb{A}^2)$ of $\Aut(\mathbb{A}^2)$ also have a non-trivial intersection which is not the case. Hence, we conclude that 
$\Aut(A_{d,2}) = \U(A_{d,2}) \rtimes \tilde{T}_{d,2}$.

Recall that  $\Aut(A_{d,2})$
 is the amalgamated product of 
 $\GL_2/\mu_d$ and $\J_d/\mu_d$ along their intersction and  $\GL_2/\mu_d = \SL_2/(\mu_d \cap \SL_2) \rtimes \tilde{T}_{d,2}$ and 
  $\J_d/\mu_d = \tilde{\J}_d/(\mu_d \cap \tilde{\J}_d) \rtimes \tilde{T}_{d,2}$. 
Since $\tilde{T}_{d,2}$ is contained in the intersection 
  $\GL_2/\mu_d \cap \J_d/\mu_d$, it follows that 
  \begin{equation}\label{formulaAd2} 
  \Aut(A_{d,2}) = \tilde{T}_{d,2} \ltimes  (\SL_2/(\mu_d \cap \SL_2) \ast_C  \tilde{\J}_d/(\mu_d \cap \tilde{\J}_d)),  
  \end{equation}
  where $C$ is the intersection of
  $A=\SL_2/(\mu_d \cap \SL_2)$ and $B=\tilde{\J}_d/(\mu_d \cap \tilde{\J}_d)$.
  As both $A$ and $B$
  are generated by unipotent subgroups it follows that  $A \ast_C B \subset \U(A_{d,2})$.   The other inclusion  follows from 
     \eqref{formulaAd2}.
 The claim follows.
\end{proof}

\begin{remark}
Define the homomorphism of abstract groups
\[ \Aut(A_{d,2}) =\tilde{T}_{d,2} \ltimes \U(A_{d,2}) \to \tilde{T}_{d,2} \]
by projection onto the first factor.  Such a homomorphism is a morphism of ind-groups which implies that $\U(A_{d,2}) \subset \Aut(A_{d,2})$ is a closed subgroup.
\end{remark}

\begin{remark}\label{evenodd}
  By Lemma \ref{amalgamAd2},
 $\U(\mathbb{A}_{d,2})$   is the amalgamated product of $\SL_2/(\mu_{d} \cap \SL_2)$ and $\tilde{\J}_{d}/(\tilde{\J}_{d} \cap \mu_{d})$  along their intersection.
 Note that if $d$ is even, then $\SL_2/(\mu_{d} \cap \SL_2)$
  is isomorphic to $\PSL_2$. If $d$ is odd, then  
  $\SL_2/(\mu_{d} \cap \SL_2)$
  is isomorphic to $\SL_2$.
\end{remark}

The following result was pointed out to me by \name{Hanspeter Kraft}.

\begin{proposition}\label{propHanspeter}
Let $Z$ be an irreducible affine normal variety of dimension $2$.

$(a)$ The groups  $\U(\SL_2 /T)$ and $\U(Z)$ are algebraically isomorphic if and only if   $Z \simeq \SL_2 /T$ or  
$Z \simeq A_{2,2}$.

$(b)$ The groups  $\U(\SL_2 /N)$ and $\U(Z)$ are algebraically isomorphic if and only if $Z \simeq \SL_2 /N$ or  
$Z \simeq A_{4,2}$.
\end{proposition}
\begin{proof}
  Let $X$ be isomorphic either to $\SL_2/T$ or to $\SL_2/N$. Then $\U(X)$ contains a copy of $\PSL_2$  (see Lemma \ref{amalgamT} and Lemma \ref{amalgamatedSL_2/N} respectively). Hence,
 by Lemma 
\ref{SL2}, $Z$ is isomorphic either to $\SL_2 /T$, to $\SL_2 /N$, or  to  $A_{d,2}$ for some  $d \in \mathbb{N}$. We claim that 
$Z$ can be isomorphic to $A_{d,2}$ 
only if $d$ is even. Indeed, assume that $\Aut(A_{d,2})$ contains an algebraic subgroup $S$ isomorphic to $\PSL_2$. Since $S \subset \U(A_{d,2})$ it follows from Lemma \ref{amalgamAd2} that   $S$ is conjugate either to a subgroup of $\SL_2/(\mu_d \cap \SL_2)$
or $\tilde{\J}_d/(\mu_d \cap \tilde{\J}_d)$ (see \cite{Sr80}). Moreover, since $\tilde{\J}_d/(\mu_d \cap \tilde{\J}_d)$ does not contain a copy of $\PSL_2$, $S$ should be conjugate to a subgroup of $\SL_2/(\mu_d \cap \SL_2)$. Hence, by
Remark \ref{evenodd} we conclude that $d$ is even.

  By Corollary \ref{corollary2} we have $\U(\SL_2 /T) \not\simeq \U(\SL_2 /N)$. Hence,
to prove $(a)$ we first need to show that an algebraic isomorphism
 $\phi\colon \U(A_{d,2}) \xrightarrow{\sim} \U(\SL_2/T)$ implies that $d=2$.
By Lemma \ref{rootseven},  the set of weights of root subgroups of $\U(A_{d,2})$ with respect to $T_{d,2}'$ is $\{\frac{kd+2}{2} | \;  k \in \mathbb{N} \cup \{ 0 \} \}$ up to an automorphism of $T'_{d,2}$.
Since $T_{d,2}$ is a subgroup of some $S \subset \U(A_{d,2})$ isomorphic to $\PSL_2$ we have by  Proposition \ref{proposition2} that the set of weights of root subgroups of $\U(X \simeq \SL_2/T)$ with respect to  $\phi(T_{d,2})$ 
is $\{ 1,2,3,\dots \}$ up to an automorphism of
$\phi(T_{d,2})$.
By Lemma \ref{rootweights}, the set of weights of root subgroups of $\U(A_{d,2})$ with respect to $T_{d,2}$ and of $\U(\SL_2/T)$ with respect to $\phi(T_{d,2})$ are equal. Therefore, $d$ indeed equals $2$. To finish the proof of $(a)$ we need 
to show that $\U(A_{2,2})$  and $\U(X \simeq \SL_2/T)$ are algebraically isomorphic.
To do so
we first note that by
Lemma \ref{amalgamAd2} and Lemma  \ref{amalgamT}, the  first factors $\SL_2/\mu_2$ and $\SO_3$  from the amalgamated product structure  of $\U(A_{2,2})$ and of $\U(X)$  respectively are isomorphic to $\PSL_2$.  
Moreover,  $\tilde{\J}_2$ and  $\J$ are algebraically isomorphic, as both $\tilde{\J}_2$ and  $\J$ are direct limits of isomorphic algebraic groups. Finally,  the intersections
$\SL_2/\mu_2 \cap \tilde{\J}_2 \subset \Aut(A_{2,2})$ and $\SO_3 \cap \J \subset \Aut(X)$ are also isomorphic  as algebraic groups as they are both isomorphic
to a Borel subgroup of $\PSL_2$.

Define a homomorphism $\varphi\colon  \U(A_{d,2}) \to \U(X)$ that sends isomorphically the first factor $\SL_2/\mu_2$ of the amalgamated product of $\U(A_{d,2})$ to the first factor $\SO_3$ of the amalgamated product of $\U(X)$ 
in a way that $\varphi(\SL_2/\mu_2 \cap \tilde{\J}_2) = \SO_3 \cap \J \subset \Aut(X)$
and 
the second factor $\tilde{\J}_2$ of the amalgamated product of $\U(A_{d,2})$ to the second factor $\J$ of the amalgamated product of $\U(X)$.  Such a map is well-defined and is an isomorphism as follows from the amalgamated product structure of  $\U(A_{2,2})$ and $\U(X \simeq \SL_2/T)$.
   The  proof of $(a)$ follows.

To prove $(b)$ we first need to show that an algebraic isomorphism
 $\phi\colon \U(A_{d,2}) \xrightarrow{\sim} \U(\SL_2/N)$ implies that $d=4$. 
 As we have already mentioned  above in this proof,  the set of weights of root subgroups of $\U(A_{d,2})$ with respect to $T_{d,2}'$ is $\{\frac{kd+2}{2} | \;  k \in \mathbb{N} \cup \{ 0 \} \}$ up to an automorphism of $T'_{d,2}$ (see  Lemma \ref{rootseven}).
 Further,
 analogously as in the first part of the proof  we have that the set of weights of  root subgroups of $\U(Y \simeq \SL_2/N)$ with respect to   $\phi(T_{d,2}')$ is $\{ 1,3,5,\dots \}$ up to an automorphism of $\phi(T_{d,2}')$ (see Corollary \ref{corollary2}).
 By Lemma \ref{rootweights}, the set of weights of root subgroups of $\U(A_{d,2})$ with respect to $T_{d,2}$ and of $\U(Y)$ with respect to $\phi(T_{d,2})$ coincide 
  which implies that $d = 4$. To finish the proof of $(b)$  we need to show that groups $\U(A_{d,2})$ and $\U(Y \simeq\SL_2/N)$ are algebraically isomorphic.
This follows analogously as in  the  previous  paragraph   in the case of
 groups $\U(A_{2,2})$  and $\U(X \simeq \SL_2/T)$.
\end{proof}

\section{Higher-dimensional case}

Consider the action of $\SL_n$ on $A_{d,n}$ induced by the standard $\SL_n$-action on $\mathbb{A}^n$.
 Denote by $S_{d,n} \subset \Aut(A_{d,n})$ the image of $\SL_n$ under the natural homomorphism $\SL_n \to \Aut(A_{d,n})$.

\begin{lemma}\label{lemma14}
 We have an isomorphism
$S_{d,n} \simeq \SL_n/\mu_{(d,n)}$,  where  $(d,n)$  denotes the greatest common divisor of $d$ and $n$. Moreover, 
$S_{d,n} \subset \U(A_{d,n})$.
\end{lemma}
\begin{proof}
 By Proposition \ref{filtration},  there is a surjective homomorphism $\phi_d \colon \Aut^{\mu_d}(\mathbb{A}^n) \to \Aut(A_{d,n})$ of groups with $\ker \phi_d = \mu_d$. Hence, 
$\Aut(A_{d,n}) \simeq \Aut^{\mu_d}(\mathbb{A}^n)/\mu_d$ which shows that $S_{d,n} \simeq  \SL_n/(\mu_{d} \cap \SL_n) \simeq \SL_n/\mu_{(d,n)}$. The second claim is clear since $S_{d,n}$ is generated by unipotent subgroups.
\end{proof}


\begin{lemma}\label{remark1}
If there is an injective algebraic homomorphism  \[
\varphi \colon S_{d,n}=\SL_n/\mu_{(n,d)} \hookrightarrow  \U(A_{l,n}),
\]then $(n, d) = (n, l)$.
In particular,
if  $\U(A_{d,n})$ and  $\U(A_{l,n})$ are algebraically isomorphic, then $(d,n) = (l,n)$.  
\end{lemma}
\begin{proof}
Applying Lemma \ref{torus} we can assume $\varphi(T_{d,n}) = T_{l,n}$. Hence, intersection $\varphi(S_{d,n}) \cap S_{l,n}$ contains $T_{l,n}$. We claim that $\varphi(S_{d,n}) = S_{l,n}$. To show this we first note that the subgroup 
$T_{l,n} \subset \varphi(S_{d,n})$ lifts to $T_n'$ and by Proposition \ref{filtration}, each root subgroup of $\varphi(S_{d,n})$ with respect to $T_{l,n}$ lifts to a one-dimensional unipotent subgroup of $\Aut(\mathbb{A}^n)$. Moreover, the subgroup $G$ of $\Aut(\mathbb{A}^n)$ generated by all one-dimensional unipotent subgroups $U_i$ lifted from root subgroups of $\varphi(S_{d,n})$ with respect to $T_{l,n}$  is algebraic subgroup of $\Aut(\mathbb{A}^n)$. Indeed, if $G$ is not algebraic, then $G$ can not be written as a finite product of $U_i$. In contrast,  $\varphi(S_{d,n})$ can be written as a finite product of root subgroups of $\varphi(S_{d,n})$ with respect to $T_{l,n}$.
Moreover, $\tilde{\phi}_d$ induces  a homomorphism of groups $G \to \varphi(S_{d,n})$ with a kernel $\mu_{(l,n)}$.

Hence,   a homomorphism of groups $G \to \varphi(S_{d,n})$ is a homomorphism of algebraic groups with the kernel $\mu_{(l,n)}$ 
and $G$ is isomorphic to $\SL_n$ that contains $T_n'$ as a maximal subtorus. It follows from \cite[Theorem 1.1]{KRZ17} that all subgroups of $\Aut(\mathbb{A}^n)$ isomorphic to $\SL_n$ are conjugate.  Therefore, $G$ is conjugate to the standard $\SL_n \subset \Aut(\mathbb{A}^n)$, i.e., there exists $\psi \in \Aut(\mathbb{A}^n)$ such that $\psi^{-1} \circ G \circ \psi = \SL_n$. Since $T_n' \subset G$ we have that $\psi^{-1} \circ T_n' \circ \psi$ is a subtorus in $\SL_n$ which implies that $\psi$ is a linear map that moreover belongs to $\GL_n$. Now it is easy to see that $G$ coincides with $\SL_n$. Therefore, $\varphi(S_{d,n}) = S_{l,n}$.

Therefore, $S_{d,n}$ is isomorphic to $S_{l,n}$ as an algebraic group. Hence, from Lemma \ref{lemma14} it follows that  $(d,n) = (l,n)$. The second part of the statement follows from the first one directly since $\U(A_{d,n})$ contains a copy of $S_{d,n}$.
\end{proof}

\begin{proposition}\label{normalcase}
Let $X$ be  $A_{d,n}$, $\SL_2/T$ or $\SL_2/N$  and  $Y$ be an irreducible affine variety.  Assume that  there is an algebraic isomorphism $\U(X) \xrightarrow{\sim} \U(Y)$. Then $\dim Y \le \dim X$. Moreover, if additionally $Y$ is normal, then

(a) if  $X \simeq \SL_2/T$, then $Y \simeq A_{2,2}$ or $Y \simeq \SL_2/T$,

(b) if  $X \simeq A_{2,2}$, then $Y \simeq A_{2,2}$ or $Y \simeq \SL_2/T$,

(c) if   $X \simeq \SL_2/N$, then $Y \simeq A_{4,2}$ or $Y \simeq \SL_2/N$,

(d) if   $X \simeq A_{4,2}$, then $Y \simeq A_{4,2}$ or $Y \simeq \SL_2/N$,

(e) if $X = A_{d,n}$, where $(d,n) \notin \{ (2,2), (2,4) \}$, $Y \simeq X$.
\end{proposition}
\begin{proof}
 Fix an algebraic isomorphism $\psi \colon \U(X)  \stackrel{\sim}{\rightarrow}  \U(Y)$ and denote by $T'$ the image of $T_{d,n}'$ if $X = A_{d,2}$ or the image of a maximal subtorus $T$ of  $\U(X)$ if $X = \SL_2/T$ or $\SL_2/N$.
 By Lemma \ref{Lie+}, Proposition \ref{proposition2} and Corollary \ref{corollary2},  all root subgroups $U \subset \U(Y)$ with respect to $T'$ have different weights. In particular, the root subgroups  $\mathcal{O}(Y)^U \cdot U \subset \U(Y)$ have different weights, which implies that $\mathcal{O}(Y)^U$ is multiplicity-free, because the map $\mathcal{O}(Y)^U \to \mathcal{O}(Y)^U \cdot U$ is injective. Hence, by Lemma \ref{Kr15},  we have that  $$\dim Y \le \dim T' + 1 = n,$$ which proves the first part of the proposition.

Now (a), (b), (c) and (d) follow from Proposition \ref{propHanspeter}.

To prove (e),  we note that   $\U(A_{d,n})$ contains a copy of $\SL_n/\mu_{(n,d)}$ which implies that $\SL_n$ acts non-trivially on $Y$ and thus, by Proposition \ref{SLn-action}, Lemma \ref{SL2} and  Proposition \ref{normalcase} (a)-(d), 
$Y \simeq A_{l,n}$ for some $l \in \mathbb{N}$. Hence, $\psi \colon \U(A_{d,n})  \stackrel{\sim}{\rightarrow}  \U(A_{l,n})$. 
By Lemma \ref{torus},  there exists an algebraic isomorphism 
$F \colon \U(A_{l,n})  \stackrel{\sim}{\rightarrow} \U(A_{l,n})$ such that 
$F( \psi(T_{d,n}'))  =  T_{l,n}'$. Therefore, we can assume that $\psi(T_{d,n}') = T_{l,n}'$.

Consider the $\mathbb{C}^+$-action $$\{ (x_1 +c x_2^{d+1},x_2,\dots,x_n) \mid c \in \mathbb{C} \} \text{ on } \mathbb{A}^n.$$ It induces the $\mathbb{C}^+$-action $U$ on $A_{d,n}$ which is  normalized
by $T_{d,n}'$. Hence, $\psi(U) \subset \U(A_{l,n})$ is a root subgroup
with respect to $\psi(T_{d,n}') = T_{l,n}'$. 
By Proposition \ref{filtration}, 
$\psi(U)$ lifts to a $\mathbb{C}^+$-action on $\mathbb{A}^n$ normalized by 
$T'_n$.
  Since $(n,d) = (n,l)$ by Lemma \ref{remark1} and because  $U$ and $\psi(U)$ have the same weight characters with respect to 
 $T_{d,n}'$ and $T_{l,n}'$ respectively (see Lemma \ref{rootweights}), 
 Lemma \ref{Lie+} implies that  $\psi(U)$ lifts to  a root subgroup
 $$\{ (x_1 +c x_2^{d+1},x_2,\dots,x_n) \mid c \in \mathbb{C} \}$$ of  $\Aut(\mathbb{A}^n)$ with respect to $T_l'$. Therefore, $l \le d$. Analogously,  $d \le l$, i.e., $d =l$. The proof follows.
\end{proof}

\begin{proof}[Proof of Theorem \ref{main4}]
Let $\psi \colon \U(X) \xrightarrow{\sim} \U(Y)$ be an algebraic isomorphism.  Proposition \ref{normalcase} implies that $\dim Y \le \dim X$.  
Since $\SL_n$ acts regularly and non-trivially on $X$,   $\SL_n$
also acts non-trivially and regularly on $Y$.

First, let  $X$ be isomorphic to $A_{d,n}$. 
Then by  Lemma \ref{SL2} and by Proposition \ref{SLn-action},   the normalization of $Y$, which we denote by $\tilde{Y}$, is isomorphic to $\SL_2/T$, $\SL_2/N$ or $ A_{l,n}$ for some $l \ge 1$. 
First, assume that $\tilde{Y} \simeq A_{l,n}$.
Hence,  Proposition \ref{SLn-action} implies that 
$$\mathcal{O}(Y) =  \sum_{i = 1,\dots,l} \bigoplus_{k \ge 0}    \mathbb{C}[x_1,\dots,x_n]_{k d_i}
$$
for some     
 $d_1,\dots,d_l \in \mathbb{N}$, where $(d_1,\dots,d_l) = l$

Let $\eta \colon A_{l,n} \to Y$  be the normalization morphism, which by Lemma \ref{closedembedding} induces the algebraic homomorphism $\tilde{\eta} \colon \U(Y)  \hookrightarrow  \U(A_{l,n})$. 
Note that $\SL_n/\mu_{(n,d)}$ acts faithfully on $X$. Then $\SL_n/\mu_{(n,d)}$ also acts faithfully on $Y$ and therefore on $A_{l,n}$. Hence, by Lemma \ref{remark1} we have that $(n,d) = (n,l)$.

Consider the $\mathbb{C}^+$-action \[
\{ (x_1 +c x_2^{d+1},x_2,\dots,x_n) \mid c \in \mathbb{C} \} \text{ on } \mathbb{A}^n.
\]
It induces the $\mathbb{C}^+$-action $U$ on $A_{d,n}$ which is  normalized
by $T_{d,n}'$. Hence, $\psi(U) \subset \U(Y)$ is a root subgroup
with respect to $\psi(T_{d,n}')$. By Lemma \ref{torus} there is an algebraic isomorphism $\U(A_{d,n})\xrightarrow{\sim} \U(A_{d,n})$ that maps $T_{d,n}'$ to  $\psi^{-1}(\tilde{\eta}^{-1}(T_{l,n}'))$ and so there is an isomorphism  $\U(Y)\xrightarrow{\sim} \U(Y)$ that maps $\psi(T_{d,n}')$ to  $\tilde{\eta}^{-1}(T_{l,n}')$. Hence, we can assume that
$\psi(U)$  is a root subgroup with respect to 
 $\tilde{\eta}^{-1}(T_{l,n}')$. 
By Lemma \ref{closedembedding}, $\psi(U)$ lifts to 
a $\mathbb{C}^+$-action on $A_{l,n}$ which is normalized by $T_{l,n}'$ and then by Lemma \ref{closedsubgroup}(c),
$\psi(U)$ lifts to a $\mathbb{C}^+$-action on $\mathbb{A}^n$ normalized by 
$T'_n$.
  Since $(n,d) = (n,l)$ and because  $U$ and $\psi(U)$ have the same weight characters with respect to 
$\tilde{\eta}^{-1}(T_{l,n}')$ and $T_{d,n}'$ respectively (i.e., if $\chi \colon T_{d,n}' \to \mathbb{C}^*$
is the weight of $\psi(U)$, then
the weight of $U$ is $\chi \circ \psi$), 
 Lemma \ref{Lie+} implies that  $\psi(U)$ lifts to  a root subgroup
 $$\{ (x_1 +c x_2^{d+1},x_2,\dots,x_n) \mid c \in \mathbb{C} \}$$ of  $\Aut(\mathbb{A}^n)$ with respect to $T_n'$.
Hence,  $$\{ (x_1 +c x_2^{d+1},x_2,\dots,x_n) \mid c \in \mathbb{C} \} \text{ induces  an action on } \mathcal{O}(\mathbb{A}^n)^{\mu_l}$$ which 
implies that $l \mid d$. Moreover, 
 $\{ (x_1 +c x_2^{d+1},x_2,\dots,x_n) \mid c \in \mathbb{C} \}$  induces an action on $\mathcal{O}(Y)$. This implies that 
 $$d + l_i \in \mathbb{N}l_1 + \dots + \mathbb{N}l_s$$
 for any $i$.
 
 The $\mathbb{C}^+$ action on $\mathbb{A}^n$ of the form
 $\{ (x_1 +c x_2^{l_i+1},x_2,\dots,x_n) \mid c \in \mathbb{C} \}$ 
induces an action on $\mathcal{O}(Y)$. Since $(d,n) = (l,n)$ it follows that  
$$\{ (x_1 +c x_2^{l_i+1},x_2,\dots,x_n) \mid c \in \mathbb{C} \} 
\text{ induces an action on } \mathcal{O}(\mathbb{A}^n)^{\mu_d}.$$
Hence, $d \mid l_i$ for any $i$ and then $d \mid (l_1,\dots,l_s) = l$.
Therefore, $d = l$. Now, because $d \mid l_i$ for any $i$, $d + l_i \in \mathbb{N}l_1 + \dots + \mathbb{N}l_s$ implies that 
$$ \mathbb{N}l_1 + \dots + \mathbb{N}l_s = \mathbb{N}_{\ge \min_i \{l_i | i =1,\dots,l \}}  d,$$ where 
$\mathbb{N}_{\ge k} = \{ m \in \mathbb{N} | m \ge k \}$.

Now assume that $\tilde{Y}$ is isomorphic to $\SL_2/T$ or to $\SL_2/N$, then by Proposition  \ref{orbits}, $Y = \tilde{Y}$. Then  (e) follows from Proposition \ref{propHanspeter}.

Let now  $X \simeq \SL_2/T$. Then by Lemma \ref{SL2}, $\tilde{Y}$ can only be isomorphic to  $\SL_2/T$, $\SL_2/N$ or 
$A_{2,2}$. 
 By Proposition \ref{propHanspeter}, $\tilde{Y}$ is isomorphic to $\SL_2/T$ or to $A_{2,2}$.
 If $\tilde{Y} \simeq  \SL_2/T$, from Proposition  \ref{orbits}, it follows that $Y = \tilde{Y}$. Hence, (b) follows from the first part of the proof. Analogously follows (d).
\end{proof}


\begin{proof}[Proof of Theorem \ref{main2}]
 The isomorphism $\Aut(X) \xrightarrow{\sim} \Aut(A_{d,n})$
 induces an algebraic isomorphism $\U(X) \xrightarrow{\sim} \U(A_{d,n})$. Note that $X$ admits a torus action of dimension $n$.  From Theorem \ref{main4} it follows that $X$ can only be isomorphic to
 $A_{d,n}^s$. On the other hand, since normalization of $A_{d,n}^s$ is equal to $A_{d,n}$, it follows from Lemma \ref{closedembedding} that there is a closed embedding 
  $\Aut(A_{d,n}^s) 	\hookrightarrow \Aut(A_{d,n})$ of ind-groups. Now the proof  follows from \cite[Proposition 9.1(3)]{RvS19}.
\end{proof}

\begin{proof}[Proof of Theorem \ref{main3}]
Let $Z$ be isomorphic either to $\SL_2/T$ or to $\SL_2/N$. Then an isomorphism $\Aut(X) \xrightarrow{\sim}  \Aut(Z)$
 induces an algebraic isomorphism $\U(X) \xrightarrow{\sim}  \U(Z)$. 
 By Theorem  \ref{main4}, $X$ is isomorphic either to $Z$ or to $A_{2k,2}^s$ for some $s \in \mathbb{N}$ and $k \in \{ 1,2 \}$. To finish the proof we need to show that $\Aut(Z)$ can not be isomorphic to $\Aut(A_{2k,2}^s)$. But this is clear as all $A_{2k,2}^s$ admit an action of a two-dimensional torus and varieties $\SL_2/T$ and $\SL_2/N$ do not admit such an
action.
\end{proof}

\vskip1cm

\end{document}